\documentclass[a4paper, 12pt]{article}
\usepackage{latexsym}
\usepackage{amssymb}
\usepackage{amsthm}
\usepackage{amsmath}

\usepackage{graphicx}

\newtheorem{theorem}{Theorem}[section]
\newtheorem{lemma}[theorem]{Lemma}

\theoremstyle{theorem}

\newtheorem{proposition}[theorem]{Proposition}
\newtheorem{corollary}[theorem]{Corollary}

\newtheorem{algorithm}[theorem]{Algorithm}

\theoremstyle{remark}
\newtheorem{remark}[theorem]{Remark}

\numberwithin{equation}{section}

\makeatletter
\def\Ddots{\mathinner{\mkern1mu\raise\p@
\vbox{\kern7\p@\hbox{.}}\mkern2mu
\raise4\p@\hbox{.}\mkern2mu\raise7\p@\hbox{.}\mkern1mu}}
\makeatother

\newcommand{\etalchar}[1]{$^{#1}$}

\def\N{\mathbb{N}}
\def\F{\mathbb{F}}
\def\C{\mathbb{C}}
\def\Z{\mathbb{Z}}

\def\Hom{\,\mbox{Hom}\,}
\def\Ind{\,\mbox{Ind}\,}

\newcommand{\conj}[1]{\overline{#1}}

\newcommand{\map}[2]{\,{:}\,#1\!\longrightarrow\!#2}

\def\G{\widetilde{G}} 
\def\K{\widetilde{K}} 
\def\B{\widetilde{B}} 
\def\T{\widetilde{T}} 
\def\H{\widetilde{H}} 

\def\I{\mathcal{I}} 

\def\a{\alpha}
\def\l{\lambda}
\def\la{\lambda}
\def\ii{{\bf i}}
\def\m{{\bf m}}
\def\La{\Lambda}

\def\p{\varpi} 

\def\gr{\mathcal{G}} 

\title{Metaplectic Whittaker Functions and Crystal Bases}
\author{Peter J. McNamara\\
\small Department of Mathematics\\[-0.8ex]
\small Massachusetts Institute of Technology, MA 02139, USA\\[-0.8ex]
\small \texttt{petermc@math.mit.edu}}

\date{March 5 2010}

\begin{document}

\maketitle

\abstract{We study Whittaker functions on nonlinear coverings of simple algebraic groups over a non-archimedean local field. We produce a recipe for expressing such a Whittaker function as a weighted sum over a crystal graph, and show that in type A, these expressions agree with known formulae for the prime power supported coefficients of Multiple Dirichlet Series.}

\section{Introduction}

Let $G$ be a simply connected Chevalley group over a non-archimedean local field, and let $U^-$ be opposite group to the unipotent radical of a Borel subgroup $B$. There are important representation-theoretic quantities that can be expressed as an integral over $U^-$, for example intertwining operators and Whittaker functions. This paper focuses on the study and evaluation of such integrals.

Specifically, for $\chi$ a character of $B$ trivial on the maximal compact subgroup $K$ of $G$, define the function $\phi_K\map{G}{\C}$ by
$$
\phi_K(bk)=(\delta^{1/2}\chi)(b)
$$ for $b\in B$ and $k\in K$, where $\delta^{1/2}$ indicates the modular quasicharacter of $B$.

Then $\phi_K$ is a $K$-invariant vector in the principal series representation $V_\chi=\Ind_B^G \delta^{1/2}\chi$. A formula for the (unique up to scaling) intertwining operator $T$ between $V_\chi$ and $V_{\chi^{w_0}}$ is given by
$$
Tf(g)=\int_{U^-} f(uw_0g)du.
$$
Here $w_0$ is the long element of the Weyl group, which acts on the set of principal series representations. It is most important to take $f=\phi_K$ and $g=1$ in the above integral.

If we let $\psi$ be a character of $U^-$, then the Whittaker function for the representation $V_\chi$ can be calculated by the integral
$$
W(g)=\int_{U^-} \phi_K(ug)\psi(u)du.
$$

The above integrals have direct analogues when one wishes to study the representation theory of central extensions $\G$ of $G$ by a finite cyclic group $\mu_n$. We shall refer to such groups as metaplectic groups, as opposed to the more restrictive notion of a metaplectic group as referring specifically to the double cover of the symplectic group. We shall review what we require from the representation theory of such groups in Section \ref{reptheory}. Despite the restriction to the reductive case in the beginning of the introduction, our work will always hold more generally for the case of a metaplectic group, and we shall work in this generality throughout.

To each reduced decomposition $w_0=s_{i_1}\ldots s_{i_N}$ of the long element of the Weyl group into a product of simple reflections, we produce a decomposition of $U^-$ into cells $C_\m$ indexed by $N$-tuples $\m=(m_1\,\ldots,m_N)$ of natural numbers, by producing an explicit version of the Iwasawa decomposition. The concept of realising a crystal combinatorially as a set of subvarieties of a unipotent radical dates back to Lusztig \cite{lusztig96}. The decomposition we consider turns out to be equivalent to that of \cite[Proposition 4.1]{kamnitzer}. However the approach taken in this paper is independent of that of Kamnitzer. The cells $C_\m$ have the property that the function $\phi_K$ is constant on each cell, so writing
$$
\int_{U^-}=\sum_{\m}\int_{C_\m}
$$
yields a combinatorial sum for these integrals we are studying, which becomes amenable to explicit calculation. In this manner, we obtain a method allowing us to evaluate our target family of integrals as a combinatorial sum over a crystal.

Indeed, there is a natural bijection between our collection of cells and the elements of the canonical basis $B(-\infty)$ of $U_q(\mathfrak{n}^+)$, the positive part of the quantised universal enveloping algebra. This connection with the combinatorics of crystals is studied in Section \ref{xtal} in the context of a positive characteristic local field. In particular, Theorem \ref{maincrystalthm} provides an explicit identification between the parametrisation of our cells and Lusztig's parametrisation of the canonical basis \cite{lusztigbook}. At the same time, we are able to relate our cell decomposition with previously constructed geometric models of crystals, in particular the realisation of the crystal in terms of Mirkovic-Vilonen cycles in the affine Grassmannian as in \cite{bravermangaitsgory}.

In Section \ref{initial}, we are able to evaluate the integral for the intertwining operator in full generality, proving a metaplectic version of the Gindikin-Karpelevich Formula; this is the content of Theorem \ref{gk}, itself a generalisation of \cite[Proposition I.2.4]{kp}. In type $A$, this presentation of the Gindikin-Karpelevich formula as a sum over a crystal has been independently obtained by Bump and Nakasuji \cite{bumpnakasuji} via a different method.

In the case of Whittaker functions, to achieve explicit results, we restrict ourselves to working in type A with a particular choice of long word decomposition. We are then able to compute in Theorem \ref{main} the metaplectic Whittaker function as a weighted sum over a crystal. We note that this weighted sum agrees exactly with the prime-power supported coefficients of a Weyl group Multiple Dirichlet Series, which we will refer to as the `$p$-part.' These multiple Dirichlet series are certain Dirichlet series in several complex variables, satisfying a set of functional equations indexed by a Weyl group.

Weyl group multiple Dirichlet series were first introduced in \cite{fiveauthor}. These are global objects built from local components, namely their $p$-parts for each prime $p$. A combinatorial  description of the $p$-part of such a series as a weighted sum over Gelfand-Tsetlin patterns is given in \cite{wgmdsxtal}. We give a full description of these coefficients at the beginning of Section \ref{tok}. In fact this particular description is given an alternative description in terms of paths in a crystal graph in \cite{wgmdsxtal}, which gives a natural interpretation of the integers $e_{i,j}$ that appear in the formulae.

This particular description was obtained in \cite{wgmdsxtal} by considering Whittaker coefficients of global metaplectic Eisenstein series attached to a minimal parabolic subgroup. The familiar unfolding technique may be applied in this metaplectic setting to show that the local components of Fourier coefficients of such an adelic metaplectic Eisenstein series are indeed given by a local Whittaker function. There is some subtlety with respect to the fact that multiple Dirichlet series do not admit an Euler product, but instead have coefficients which satisfy a weaker property of twisted multiplicativity.

Viewed in this regard, our results can be considered as an alternative approach to the combinatorial formulae appearing in \cite{wgmdsxtal}. In fact, in section \ref{tok}, we detail our calculation in type A for a standard type of long word, exactly producing the combinatorial formulae of \cite{wgmdsxtal} in Theorem \ref{main}.

Because of the generality of our approach, which in principle can be used for arbitrary root systems and long word decompositions, our method may be viewed as providing a recipe for writing these Whittaker functions in the form of a generating function indexed by a crystal, as well as providing some understanding as to why the combinatorics of crystal graphs make an appearance in this field.

Independently, there has been an alternative calculation of the metaplectic Whittaker function in type A by Chinta and Offen \cite{chintaoffen}. This result shows that the spherical Whittaker function agrees with the $p$-parts of Weyl group multiple Dirichlet series using a different definition of the $p$-part, due to Chinta and Gunnells \cite{chintagunnells}. The results of this paper, combined with the work of Chinta and Offen provide a resolution of the question of proving that these two differing definitions of multiple Dirichlet series agree, which hitherto had been an open problem.

The author would like to thank B. Brubaker for his guidance through this area of study, and his assistance in the preparation of this manuscript.


\section{Preliminaries}\label{prelim}

In this section, we set up notation for use in the sequel.

Let $\Phi$ be a reduced root system, $\Phi^+$ and $\Phi^-$ a choice of positive and negative roots respectively, and $I$ the finite index set of simple roots. We denote by $\Phi^\vee$ the coroots and $\a\mapsto \a^\vee$ the bijection between $\Phi$ and $\Phi^\vee$. We use  $\langle\cdot,\cdot\rangle\map{\Phi\times\Phi^\vee}{\Z}$ to denote the canonical pairing between $\Phi$ and $\Phi^\vee$.

Let $W$ be the Weyl group of the root system $\Phi$. It is generated by simple reflections $s_i$ for each $i\in I$. Let $w_0$ denote the element of longest length in $W$ and denote that length by $N$. The notation $[N]$ will be used for the set $\{1,2,\ldots,N\}$.

Let $F$ be a field equipped with a non-trivial discrete valuation $v$ and compatible norm $|\cdot |$. We shall be mostly concerned with the case of a local field, which implies that the cardinality of the residue field is finite. We use $q$ to denote this cardinality and normalise $|\cdot|$ such that $|x|=q^{-v(x)}$. Let $O_F$ be the valuation ring of $F$ and $\p$ be a uniformiser, that is a generator of the maximal ideal of $O_F$. In order to establish the connection with crystal graphs in Section \ref{xtal}, we restrict to the case where $F$ is the field of Laurent series over an algebraically closed field, although the rest of the paper is concerned with the case where $F$ is a local field.

Let $n$ be a positive integer not divisible by the residue characteristic and such that $F$ contains $2n$ $2n$-th roots of unity. Let $(\cdot,\cdot)\map{F^\times \times F^\times}{\mu_n}$ be the $n$-th power Hilbert symbol. The properties of the Hilbert symbol that we shall explicitly use are its bilinearity and skew symmetry, so for example $(xy,z)=(x,z)(y,z)$ and $(x,y)(y,x)=1$. The deeper properties of the Hilbert symbol will only be used implicitly, in that they are necessary in order to construct the central extension $\G$ of $G$.

Associated to the root system $\Phi$, there exists a split, simple, connected, simply connected algebraic group over $F$, which we denote by $G$. Let $T$ be a maximal torus of $G$, $X$ be the character group of $T$, and $Y$ the cocharacter group. Then we have $Y=\Z\Phi^\vee$, while $\Z\Phi$ is a finite index sublattice of $X$. We extend $\langle\cdot,\cdot\rangle$ to a perfect pairing $\langle\cdot,\cdot\rangle\map{X\times Y}{\Z}$ by linearity. We also will require the corresponding adjoint group $G^\text{ad}$ and the cocharacter group $\La$ of the corresponding torus $T^\text{ad}$. This cocharacter group $\La$ naturally contains the coroot lattice $Y$ as a finite index sublattice.

Choose a Borel subgroup $B$ of $G$ containing $T$ and let $U$ denote its unipotent radical. We focus most of our attention on the group $U^-$, the opposite unipotent subgroup to $U$.

Let $\psi$ denote an additive character of $F$ with conductor $O_F$. Thus $\psi$ is a homomorphism from the additive group $F$ to $\C^\times$ that is trivial on $O_F$ and non-trivial on $\frac{1}{\p}O_F$.

We will encounter Gauss sums appearing in our work, so we shall pause to define them and recap their properties. We define the Gauss sum \begin{equation}\label{gausssumdefn}
g(a,b)=\int_{O_F^\times}(u,\p)^a\psi(\p^bu)du.
\end{equation}
where we consider $du$ to mean additive measure on $F$, normalised such that $O_F^\times$ has volume $q-1$. This choice of normalisation of the Haar measure ensures that this approach to defining the Gauss sum agrees precisely with the classical definition as a sum over a finite field. We choose this approach since these precise integrals will appear later in this work. The following proposition provides a list of standard properties of these Gauss sums whose proofs involve routine manipulation of the defining integral.

\begin{proposition} \label{gauss}  Gauss sums satisfy the following identities.
\begin{enumerate}
\item For all $a$ and $b$, $\conj{g(a,b)}=(-1,\p)^ag(-a,b)$.
\item If $b<-1$, then $g(a,b)=0$.
\item If $b\geq 0$, then $g(a,b)=q-1$ if $a$ is divisible by $n$, and zero otherwise.
\item If $n$ divides $a$, then $g(a,-1)=-1$.
\item If $a$ is not divisible by $n$, then $|g(a,-1)|=q^{1/2}$.
\end{enumerate}
\end{proposition}

We make a note about our usage of the symbol $\prod$ in the noncommutative case: Either the terms in the product will all commute, or we will write $\prod_{k=m}^n x_k$ for $x_m x_{m+1}\ldots x_n$ and $\prod_{k=n}^m x_k$ for $x_n x_{n-1},\ldots x_m$ where $m\leq n$.

\section{Central Extensions of Chevalley Groups}\label{1}
In this section, and any other section explicitly relying upon central extensions, we assume that $F$ is a non-archimedean local field, so that we have a good theory of the Hilbert symbol at our disposal. It will only be later when the connection to crystals is studied that we shall wish to work over other discrete valuation fields.

Since the group $G$ is equal to its commutator subgroup, it admits a universal central extension $E$ \cite{moore}. Thus we have a short exact sequence of groups
\[
1\rightarrow A\rightarrow E\rightarrow G\rightarrow 1
\]
with $A$ in the centre of $E$. Steinberg gives a presentation of $E$ in terms of generators and relations, which we now quote:

\begin{theorem}[Theorem 10, \cite{steinbergyale}]\label{relations} 
The group $E$ is generated by symbols $e_\alpha(x)$ where $\alpha\in\Phi$ and $x\in F$, subject to the relations
\[
e_\alpha(x)e_\alpha(y)=e_\alpha(x+y)
\]
\[
w_\alpha(x)e_\alpha(y)w_\alpha(-x)=e_{-\alpha}(-x^{-2}y)
\] where $w_\alpha(x)=e_\alpha(x)e_{-\alpha}(-x^{-1})e_\alpha(x)$, and
\begin{equation}\label{eecommute}
e_\alpha(x)e_\beta(y)=\big[\!\!\!\prod_{\substack{ i,j\in\Z^+ \\ i\alpha+j\beta=\gamma\in\Phi }} \!\!\!e_\gamma(c_{i,j,\alpha,\beta}x^iy^j)\big]e_\beta(y)e_\alpha(x)
\end{equation} for all $x,y\in F$ and $\a,\beta\in\Phi$ with $\a+\beta\neq 0$, where $c_{i,j,\alpha,\beta}$ is a fixed collection of integers, completely determined by the root system $\Phi$.
\end{theorem}

Note that the product in (\ref{eecommute}) is a product of commuting terms, so there is no possible ambiguity present with respect to order of multiplication.

Matsumoto \cite{matsumoto} gave a computation of the kernel $A$, and showed it to be equal to $K_2(F)$ except in type $C$ in which case $K_2(F)$ is canonically a quotient group of $A$. Here $K_2(F)$ denotes the algebraic K-theory of $F$ (in degree 2).

As $F$ is a non-archimedean local field containing $n$ $n$-th roots of unity, the Hilbert symbol gives a surjection from $K_2(F)$ to $\mu_n$. Thus in all cases, we have a surjection $A\rightarrow \mu_n$, so we can take the quotient of $E$ by the kernel of this surjection to obtain a central extension of $G$ by $\mu_n$, which we shall denote by $\G$. Let $p$ denote the quotient map from $\G$ to $G$. So we have the exact sequence
$$
1\rightarrow \mu_n\rightarrow\G\xrightarrow{p} G\rightarrow 1.
$$

From now on, we shall continue to use the notation $e_\alpha(x)$ and $w_\alpha(x)$ to refer to the corresponding elements of $\G$. These elements naturally satisfy all the relations in Theorem \ref{relations}. Also for any subgroup $H$ of $G$, we shall denote by $\H$ the induced covering group of $H$.

Define the elements $h_\a(x)\in\G$ by $h_\a(x)=w_\a(x)w_\a(-1)$ and let $l_\alpha=||\a^\vee||^2$ where the norm on the coroot lattice has been chosen such that the short coroots have length 1. Then the following identities hold in $\G$ \cite[\S6]{steinbergyale}.

\begin{equation}\label{ehcommute}
h_\alpha(x)e_\beta(y)h_\alpha(x)^{-1}=e_\beta(x^{\langle\beta,\alpha^\vee\rangle}y)
\end{equation}
\begin{equation}
h_\a(x)h_\a(y)=(x,y)^{l_\a}h_\a(xy)
\end{equation}
\begin{equation}\label{hhcommute}
h_\alpha(x)h_\beta(y)h_\alpha(x)^{-1}=h_\beta(y)(x,y)^{\langle\beta,\alpha^\vee\rangle l_\beta}
\end{equation}
\begin{equation}\label{inverseh}
h_\a(x)=h_{-\a}(x^{-1}).
\end{equation}
Recall that $(x,y)\in\mu_n$ in the above equations is the value of the Hilbert symbol and is central in $\G$.

\section{An Explicit Iwasawa Decomposition}
In this section, we shall give an algorithm that explicitly calculates the Iwasawa decomposition in $\G$. This algorithm will form the cornerstone for the rest of this paper.

Let $K=G(O_F)$, a maximal compact subgroup of $G$. Since we have assumed that $n$ and $q$ are coprime, the central extension $\G$ of $G$ splits over $K$ \cite[Lemma 11.3]{moore}. Thus there is a section $s\map{K}{\G}$ that is a homomorphism such that $\K$ is the direct product of its subgroups $s(K)$ and $\mu_n$. We shall choose such a splitting $s$ and identify $K$ with its image in $\G$ under $s$ when appropriate.

Let $T$ be the subgroup of $G$ generated by (the images of) all elements of the form $h_\alpha(x)$, and let $B$ be the subgroup of $G$ generated by $T$ and the images of $e_\alpha(x)$ for all $\alpha>0$ and $x\in F$.  Then $B$ is a Borel subgroup of $G$ with maximal torus $T$. The Iwasawa decomposition \cite{bruhattits} states that $G=BK$. This clearly lifts to $\G$, and so we may write $\G=\B K$.

The unipotent subgroup opposite to $B$ in $G$ is denoted $U^-$. Steinberg \cite{steinbergyale} shows that the universal central extension $E$ splits over $U^-$, so our central extension $\G$ certainly also does. Thus we may identify $U^-$ with a corresponding subgroup of $\G$. Explicitly, we have that $U^-$ is the subgroup of $\G$ generated by all $e_\a(x)$ where $\a\in\Phi^-$ and $x\in F$. We shall give an algorithm that explicitly calculates the Iwasawa decomposition of any $u\in U^-$.

\begin{lemma}\label{unipotentcompactsplitting}
If $\a\in\Phi$ and $x\in O_F$, then $e_\a(x)\in s(K)$.
\end{lemma}

\begin{proof}
Let us fix one $\a\in\Phi$. The group $\{e_\a(t)|t\in O_F\}$ is a splitting of the corresponding subgroup of $K$, a subgroup that is isomorphic to the additive group $O_F$. Any two splittings of this subgroup must differ by an element of $\Hom(O_F,\mu_n)$, which is trivial, so the splitting is unique, and thus must agree with the restriction of $s$, proving that $e_\a(t)\in s(K)$.
\end{proof}

We shall now recall a result regarding the action of the Weyl group on a root system.

\begin{proposition}\label{ellw}
Let $w=s_{i_1}\ldots s_{i_k}$ be a reduced decomposition of $w\in W$ into simple reflections $s_{i_j}$. Then
\[
\{\alpha\in \Phi^+\mid w(\alpha)\in \Phi^-\}=\{\alpha_{i_k},s_{i_k}\alpha_{i_{k-1}},s_{i_k}s_{i_{k-1}}\alpha_{i_{k-2}},\ldots,s_{i_k}s_{i_{k-1}}\ldots s_{i_2}\alpha_{i_1}\}.
\]
\end{proposition}
The proof may be found in any standard text on Lie theory, for example \cite[Ch VI, \S 6, Corollarie 2]{bourbaki}.

Let $\ii=(i_1,\ldots,i_N)$ be an $N$-tuple of elements of $I$ such that $w_0=s_{i_1}s_{i_2}\ldots s_{i_N}$ is a reduced decomposition of the long word $w_0$. Let $\I$ denote the set of all such tuples $\ii$.

For any $\ii\in\I$, the above theorem induces a total ordering $<_\ii$ on the set of positive roots, given by
$$\Phi^+=\{
\alpha_1<_\ii\cdots<_\ii\alpha_N \}\quad\text{   where   }\quad\alpha_j=s_{i_N}\ldots s_{i_{j+1}}\alpha_{i_j}.$$
With this notation, we have now defined $\a_j$ for $j\in I$ as well as for $j\in [N]$. We hope that this does not cause any confusion for the reader.

For each $k\in[N]$, let $G_k$ denote the set of elements $g\in\G$ which can be expressed in the form
\begin{equation}\label{gkdefn}
g=\left(\prod_{j=1}^{k-1}e_{-\alpha_j}(x_j)\right)h\left(\prod_{j=k+1}^{N}e_{\alpha_j}(x_j)\right)
\end{equation}
where all $x_j\in F$ and $h\in \T$.

\begin{lemma}\label{lemma1}
For all $z\in F$ and $g\in G_k$, there exists unique $z'\in F$ and $g'\in G_k$ such that
\[
e_{-\alpha_k}(z)g=g'e_{-\alpha_k}(z').
\]
\end{lemma}
\begin{proof}
This involves repeated applications of the commutation relations presented in Section \ref{1} to push the $e_{-\a_k}(z)$ term in $e_{-\a_k}(z)g$ from the left to the right hand side of the product. It is straightforward to push a term of the form $e_{-\a_k}(\cdot)$ past another term of the form $e_{-\a_{k'}}(\cdot)$ or past a term of the form $h_\beta(\cdot)$. The key to ensuring that the manipulation claimed by the lemma is possible is thus understanding what terms appear in the commutator formula (\ref{eecommute}) when a term of the form $e_{-\a}(x)$ is pushed past a term of the form $e_\beta(y)$ for some $\a,\beta\in\Phi^+$, necessarily with $\a<_\ii\beta$.

Suppose that a term of the form $e_\gamma(w)$ appears in computing such a commutator. We will show that if $\gamma\in\Phi^+$, then $\gamma>_\ii\beta$ whereas if $\gamma\in\Phi^-$, then $-\gamma<_\ii\a$. This is enough to see that after applying the commutation relations a finite number of times, we will have written $e_{-\a_k}(z)g$ in the desired form $g'e_{-\a_k}(z')$. Briefly, the reason that this suffices is the following. First, we push the $e_{-\a_k}(\cdot)$ term to the rightmost side of the product. Then consider the multiset consisting of all integers $j-i$ for which there is a $e_{-\a_i}(\cdot)$ term appearing to the left of a $e_{\a_j}(\cdot)$ in our product. Such occurrences are undesirable. With each application of the commutation relations, an element $j-i$ in this multiset is replaced by a finite set of integers strictly larger than $j-i$. Since the elements of this multiset are integers bounded above by $N$, eventually our multiset becomes empty, as desired.

So now let us suppose that $\gamma$ is such that a term $e_\gamma(w)$ appears in pushing a term of the form $e_{-\a}(x)$ is past a term of the form $e_\beta(y)$ where $\a,\beta\in\Phi^+$ with $\a<_\ii\beta$.. Then we have $\gamma=i\beta-j\a$ for some positive integers $i$ and $j$. Write $\a=s_{i_N}\ldots s_{i_{r+1}}\a_{i_r}$ and $\beta=s_{i_N}\ldots s_{i_{s+1}}\a_{i_s}$. Note that by our assumption $\a<_\ii\beta$ we have $r<s$.

Let $w_1=s_{i_{r+1}}\ldots s_{i_N}$. Then using Proposition \ref{ellw} we obtain that $w_1(\beta)\in\Phi^-$ and $w_1(\a)\in\Phi^+$. Thus if $\gamma\in\Phi^-$, we have $-\gamma\in\Phi^+$ and $w_1(-\gamma)\in\Phi^+$ so we can conclude, again using Proposition \ref{ellw} that $-\gamma<_\ii\a$ as desired.

Now let $w_2=s_{i_{s-1}}\ldots s_{i-1}$. Note that $-w_0\a=s_{i_1}\ldots s_{i_{r-1}}\a_r$ and similarly for $\beta$. Thus we use Proposition \ref{ellw} for $w_2$ to obtain that $w_2(-w_0\a)\in\Phi^-$ and $w_2(-w_0\beta)\in\Phi^+$ so $w_2(-w_0\gamma)\in\Phi^+$. Since $\gamma\in\Phi^+$ if and only if $-w_0\gamma\in\Phi^+$ this last result implies that we can write $-w_0\gamma=s_{i_1}\ldots s_{i_{t-1}}\a_t$ for some $t>s$, implying $\gamma=s_{i_N}\ldots s_{i_{t+1}}\a_t$ and thus $\gamma>_\ii\beta$, as required.

\end{proof}

Now we are in a position to describe our algorithm for an explicit Iwasawa decomposition of $u\in U^-$. A sample computation in the case of $SL_3$ will be presented at the conclusion of this section. This algorithm runs as follows:

\begin{algorithm}\label{algorithm}

Given $u\in U^-$, we may write $u$ in the form
\[
u=\prod_{j=1}^Ne_{-\alpha_j}(x_j).
\]
for unique $x_1,\ldots,x_N \in F$.

For $k\in [N]$, we inductively define elements $p_k,p'_k\in\G$ and $y_k\in F$ as follows:

 Initialise $p_{N+1}=1_{\G}$. By decreasing induction on $k$, we use Lemma \ref{lemma1} to define $p'_k\in G_k$ and $y_k\in F$ such that $e_{-\alpha_k}(x_k)p_{k+1}=p'_ke_{-\alpha_k}(y_k)$.

Let \[
p_k=\begin{cases}p'_k &\text{if  } |y_k|\leq 1 \\
p'_kh_{\a_k}(y_k^{-1})e_\a(y_k) &\text{if  } |y_k|> 1. \end{cases}
\]

Once this computation has been completed for $k=1$, the algorithm halts, with its primary output as the group element $p_1$.
\end{algorithm}

In the above algorithm, variables $x_i$ and $y_i$ are introduced which we shall make frequent use of throughout the sequel. Both of these sets of variables are implicit functions of $u$ and $\ii$, however for ease of exposition, this dependence will be usually suppressed from the notation.

\begin{theorem}
The above algorithm produces a well defined output that computes the Iwasawa decomposition of $u\in U^-$, in the sense that writing $u=p_1k'$, we have $p_1\in\B$ and $k'\in K$.
\end{theorem}

\begin{proof}
It is a consequence of Lemma \ref{lemma1} that at all steps in the inductive procedure, $p_k\in G_k$ and $y_k\in F$ are uniquely determined. This is mostly straightforward, the only potential sticking point is that a priori we do not have $p_{k+1}\in G_k$. Instead we have $p_{k+1}=e_{-\a_k}(z_k)p''$ for some $z_k\in F$ and $p''\in G_k$ and thus writing $e_{-\a_k}(x_k)p_{k+1}=e_{-\a_k}(x_k+z_k)p_k''$ we see that the algorithm may continue unhindered. Thus the output of this algorithm is well defined.

Note that $p_1\in\B$ since $p_1'\in G_1$.
To show that $k'\in  K$ we need to show that if $|y_k|<1$ then $e_{-\a_k}(y_k)\in K$, and secondly if $|y_k|\geq 1$ then $e_{-\a_k}(y_k)^{-1}h_{\a_k}(y_k^{-1})e_\a(y_k)\in K$, as $k'$ is a product of such terms. Since we have an identity
$$
e_{-\a}(y)^{-1}h_\a(y^{-1})e_\a(y)=e_\a(y^{-1})w_{-\a}(1),
$$
this desired result follows from Lemma \ref{unipotentcompactsplitting}.
\end{proof}

Let us define variables $w_k$ for $k\in[N]$ as follows
$$ w_k=\begin{cases}y_k &\text{if  } |y_k|> 1 \\
1 &\text{if  } |y_k|\leq 1. \end{cases} $$
Then the following description of the diagonal part of $p_k$ is immediate.
\begin{proposition}\label{diagonalterms}
When writing $p_k$ in the form
\[
p_k=\left(\prod_{j=1}^{k-1}e_{-\alpha_j}(t_j)\right)h\left(\prod_{j=k+1}^{N}e_{\alpha_j}(t_j)\right),
\] with $h\in \T$, we have
$$ h=\prod_{i=N}^k h_{-\a_i}(w_i).$$
\end{proposition}

\noindent Here we are freely using the fact that $h_{-\a}(w)=h_\a(w^{-1})$.

Define integers $m_k$ for $k\in[N]$ according to the valuation of the variables $w_k$ by $m_k=-v(w_k)$. (The $m$ variables, like $x$, $y$ and $w$ are functions of $u$ and $\ii$, again this is usually suppressed from the notation). We are now in a position to define the subsets of $U^-$ that will be of primary concern to us.

Given any $\ii\in\I$ and $\m=(m_1,\ldots,m_N)\in \N^N$, define
$$
C^\ii_\m=\{u\in U^-\mid m_k(u,\ii)=m_k\text{ for all }k\in[N]\}.
$$
For linguistic convenience, we shall at times refer to these sets as `cells.' For a fixed choice of $\ii$, the set of above cells clearly decompose $U^-$ as a disjoint union.

We now give an example of these cells in the case where $G=SL_3$. Factorise $u\in U^-$ as
\[
u=\begin{pmatrix}
  1 & 0 & 0 \\
  x & 1 & 0 \\
  y & z & 1
\end{pmatrix}
=\begin{pmatrix}
  1 & 0 & 0 \\
  x & 1 & 0 \\
  0 & 0 & 1
\end{pmatrix}
\begin{pmatrix}
  1 & 0 & 0 \\
  0 & 1 & 0 \\
  y & 0 & 1
\end{pmatrix}
\begin{pmatrix}
  1 & 0 & 0 \\
  0 & 1 & 0 \\
  0 & z & 1
\end{pmatrix}
\] where $x,y,z\in F$ and introduce the auxiliary variable $w=xz-y$. Then for a triple of non-negative integers $\m=(m_1,m_2,m_3)$, the cell $C^{212}_{\m}$ is given by the set of $u$ as above subject to the following conditions:
\[\begin{array}{lccc}
\mbox{If  }\ m_1,m_2,m_3>0: & v(z)=-m_3 & v(\frac{y}{z})=-m_2, & v(\frac{wz}{y})=-m_1. \\
\mbox{If  }\ m_1=0;m_2,m_3>0: & v(z)=-m_3, & v(\frac{y}{z})=-m_2, & |\frac{wz}{y}|\leq 1. \\
\mbox{If  }\ m_2=0;m_1,m_3>0: & v(z)=-m_3, & |\frac{y}{z}|\leq 1, & v(w)=-m_1. \\
\mbox{If  }\ m_2=m_1=0;m_3>0: & v(z)=-m_3, & |\frac{y}{z}|\leq 1, & |w|\leq 1. \\
\mbox{If  }\ m_3=0;m_1,m_2>0: & |z|\leq 1, & v(y)=-m_2, & v(\frac{x}{y})=-m_1. \\
\mbox{If  }\ m_3=m_1=0;m_2>0: & |z|\leq 1, & v(y)=-m_2, & |\frac{x}{y}|\leq 1. \\
\mbox{If  }\ m_3=m_2=0;m_1>0: & |z|\leq 1, & |y|\leq 1, & v(x)=-m_1. \\
\mbox{If  }\ m_3=m_2=m_1=0: & |z|\leq 1, & |y|\leq 1, & |x|\leq 1.
\end{array} \]



\section{Representation Theory and Metaplectic Whittaker Functions}\label{reptheory}

We require some knowledge of unramified genuine principal series representations of $\G$. This is treated in \cite{savin} in the simply laced case, and in \cite{mcn} in greater generality that encompasses all our needs. One can also consult \cite{kp} for a detailed study in the case $G=GL_n$. The meaning of the adjective `genuine' is that the central $\mu_n$ must act by a fixed faithful character. We fix an inclusion $\mu_n\subset\C^\times$ and thus assume that for any genuine representation $(\pi,V)$ of $\G$, we have that $\pi(\zeta g)=\zeta\pi(g)$ for all $\zeta\in\mu_n$ and $g\in\G$.

The group $\T$ is a Heisenberg group (a two-step nilpotent group). Since the commutator subgroup $[\T,\T]$ is contained in the central $\mu_n$, any irreducible genuine representation of $\T$ is induced from a character of a maximal abelian subgroup of $\T$ \cite[Corollary 5.2]{mcn}.

The group $\T\cap K$ is abelian since the Hilbert symbol is trivial on $O_F^\times\times O_F^\times$. Let $H$ be a maximal abelian subgroup of $\T$ containing $\T\cap K$.

Let $\chi$ be a genuine complex character of the group $H/(\T\cap K)$. We choose complex numbers $x_i$ for $i\in I$ with the property that $$
\chi\left(\prod_{i\in I}h_{\a_i}(\p^{m_i})\right)=\prod_{i\in I}x_i^{m_i}
$$
whenever the argument of $\chi$ above lies in $H$. This is always possible, the condition on $\{m_i\}_{i\in I}$ for the above product to lie in $H$ is that $m_i$ lie in a sublattice of $n\Z^I$. An explicit characterisation of this lattice is obtained in \cite{mcn} though we shall not require such a description here.

Let $(\pi_\chi,W_\chi)$ be the irreducible representation of $\T$ obtained by inducing $\chi$ from $H$ to $\T$. We obtain an unramified principal series representation $(\pi_\chi,V_\chi)$ of $\G$ by taking the normalised induced representation of $W_\chi$, considered as a representation of $\B$ (since $\T$ is a quotient of $\B$, we may inflate $W_\chi$ to a representation of $\B$). Explicitly, $V_\chi$ is realised as the space of functions $f\map{\G}{W_\chi}$ that are right invariant under an open compact subgroup and
$$
f(bg)=\delta^{1/2}(b)\pi_\chi(b)f(g)
$$
for all $b\in\B$ and $g\in\G$ where $\delta$ is the modular quasicharacter of $\B$. The action of $\G$ on $V_\chi$ is given by right translation.

The following theorem is largely contained in the works of Kazhdan and Patterson \cite{kp} and Savin \cite{savin}. The work \cite{mcn} contains the added generality necessary to state this result precisely.

\begin{theorem}
An unramified genuine principal series representation $(\pi_\chi,V_\chi)$ as above has a one-dimensional space of $K$-fixed vectors.
\end{theorem}

Let $\phi_K\map{\G}{W_\chi}$ be a non-zero element of $V_\chi^K$. We define a function $f\map{\G}{\C}$ related to $\phi_K$ by the following.
\begin{equation}\label{fdefn}
f\left(\zeta u\prod_{i\in I}h_{\alpha_i}(\p^{m_i})k\right)=\zeta\prod_{i\in I}(q^{-1}x_{\alpha_i})^{m_i}.
\end{equation} where $\zeta\in\mu_n$, $u\in U$, $m_i\in\Z$ and $k\in K$, using the Iwasawa decomposition to express an arbitrary $g\in\G$ as a product of such terms.

It is not immediately obvious that this is a well-defined function; we shall prove this as part of the following Propostion.

\begin{proposition}
There exists a linear functional $\lambda$ on $W_\chi$ such that $\lambda(\phi_K(g))=f(g)$ for all $g\in\G$.
\end{proposition}


\begin{proof}
Let $A$ be the group of all elements of the form $\prod_{i\in I} h_{\a_i}(\p^{m_i})$ where $m_i\in\Z$. This is abelian because our assumption that $\mu_{2n}\subset F$ forces all commutators calculated from (\ref{hhcommute}) to vanish.

Choose a set of coset representatives $a_j\in A$ for $\T/H$. Then the set of elements $\{\pi_\chi(a_j)\phi_K(1)\}$ forms a basis of $W_\chi$. Thus there exists a linear functional $\lambda\map{W_\chi}{\C}$ such that $$\lambda(\pi_\chi(a_j)\phi_K(1))=\chi(a_j)$$ where $\chi\map{A}{\C}$ is the extension of $\chi$ on $A\cap H$ given by $\chi(\prod_{i\in I}h_{\a_i}(\p^{m_i}))=\prod_{i\in I}x_i^{m_i}$.

Now suppose we have an arbitrary $a\in A$. Then we can write $a=a_j h$ for some $j$ and some $h\in H$ and hence \begin{eqnarray*}\lambda(\phi_K)(a)&=&\lambda(\pi_\chi(a_j)\phi_K(h)) \\
&=&(\delta^{1/2}\chi)(h)\lambda(\pi_\chi(a_j)\phi_K(1)) \\
&=&(\delta^{1/2}\chi)(h)
(\delta^{1/2}\chi)(a_i) \\
&=&(\delta^{1/2}\chi)(a).\end{eqnarray*}

 Thus $\lambda\circ\phi_K$ and $f$ agree on $A$. Since $\phi_K$ is a genuine function, left $U$-invariant and right $K$-invariant, this is enough to show that $\lambda\circ\phi_K=f$, thus proving the proposition, and as a consequence ensuring that the function $f$ is well-defined.
\end{proof}


Our main thrust of this paper involves computing the integral of $f$ over $U^-$, as well as computing the integral of $f$ against a character of $U^-$. We shall now explain the representation theoretic meanings behind these integrals.

The Weyl group $W$ acts on principal series representations. As in the reductive case, we have \cite{kp,mcn}
$$
\dim(\Hom(V_\chi,V_\chi^{w_0}))=1.
$$
Letting $T$ denote a non-zero element of the above space of intertwining operators, and $\phi_K$ and $\phi'_K$ be spherical vectors for $V_\chi$ and $V_\chi^{w_0}$ respectively, we have $T\phi_K=g(\chi)\phi_K'$ for some meromorphic function $g$. The representation $V_\chi$ will be irreducible whenever $g$ does not have a zero or a pole, and we can write $g$ as the integral
\begin{equation*}
 \int_{U^-} f(u)du.
\end{equation*} with $f$ as in (\ref{fdefn}) above.
It will transpire that this integral is only dependent on $\chi$, not on the choices of $x_i$ that were made in defining $f$.

Now let $\psi$ denote a character of $U^-$ such that the restriction to the subgroup $U_{-\a}=\{e_{-\a}(x)\mid x\in F\}\cong F$ for each simple $\a$ has conductor $O_F$. As in the reductive case, the integral
\[
W(g)= \int_{U^-} \phi_K(ug)\psi(u)du
\]
defines a Whittaker function; it is a spherical vector in a Whittaker model for $V$. The above integral actually defines a $W_\chi$-valued Whittaker function, so to obtain a complex valued Whittaker function, we merely need to compose with a linear functional $\lambda$, and thus we are lead to study the integral
\[
 \int_{U^-} f(ug)\psi(u)du.
\]
Since a Whittaker function transforms on the left under $U$ by $\psi$, on the right under $K$ trivially, and is genuine, a Whittaker function is determined by the values it takes on the maximal torus $T$, by the Iwasawa decomposition. If $g\in T$, making a change of coordinate in the above integral, replacing $u$ by $gug^{-1}$, means that one is instead lead to evaluating the integral
\[
 \int_{U^-} f(u)\psi(gug^{-1})du
\]
where a relatively trivial constant factor has been discarded.

The lattice $\La$ is the coroot lattice of the corresponding adjoint group $G^{\text{ad}}$, and thus acts by conjugation on $U^-$ and hence on the set of characters of $U^-$. Explicitly, the action of $\la\in \La$ on our fixed character $\psi$ creates another character $\psi_\la$ defined by
$$
\psi_\la(u)=\psi(\p^\la u \p^{-\la}).
$$
Here, although $\p^\la$ is not in $G$, it is in $G^{\text{ad}}$ which canonically shares the same unipotent subgroups as $G$, so acts on $U^-$ by conjugation.

In this language, explicitly evaluating a Whittaker function becomes equivalent to the problem of evaluating the integral
\begin{equation}\label{iladefn}
I_\la= \int_{U^-} f(u)\psi_\la(u)du.
\end{equation}

For the remainder of this paper, our major goal is to obtain a method for evaluating the above integral.

There is a simple case where we know that this integral vanishes. This vanishing property, together with its proof, is exactly the same as in the reductive case, as shown in \cite[Lemma 5.1]{cs}.

\begin{proposition}
Unless $\la$ is dominant, we have $I_\la=0$.
\end{proposition}

\section{Initial Calculations with the Cell Decomposition}\label{initial}
In this section we shall present some initial calculations involving the decomposition of $U^-$ into the disjoint union
$$
U^-=\bigsqcup_{\m\in\N^N} C^\ii_\m.
$$
This will immediately result in a metaplectic Gindikin-Karpelevic formula, generalising the type A results of Bump and Nakasuji \cite{bumpnakasuji}. The intermediate results we obtain will also prove to be useful in the sequel when we are evaluating the Whittaker function.

We choose a Haar measure on $F$ such that $O_F$ has volume 1, and denote it by $dx$. Now choose a normalisation of Haar measure on $U^-$ such that $du=\prod_{i=1}^N dx_i$. The variables $x_i$ here that we use are those introduced in Algorithm 1, we also freely use the $y_i$ and $w_i$ variables defined in the main algorithm (which all depend on a choice of $\ii$). Under this normalisation, we can now compute the volume of the sets $C_\m^\ii$.

\begin{proposition}
The volume of $C_\m^\ii$ is
\[
\prod_{i=1}^N q^{\langle\rho,\alpha_i^\vee\rangle m_i}(1-\frac{1-\delta_{m_i,0}}{q})
\] ($\delta$ here represents the Kronecker delta function).
\end{proposition}

\begin{proof}
We wish to calculate $\int_{C^\ii_\m} 1\ dx_1\ldots dx_N$ and it is easy to see that
$$
\int_{C^\ii_\m} 1\ dy_1\ldots dy_N =\prod_{i=1}^N q^{m_i}(1-\frac{1-\delta_{m_i,0}}{q}).
$$
So we need to calculate the Jacobian for the change of coordinates from $x_1,\ldots,x_N$ to $y_1,\ldots,y_N$, the formulae for which are obtained as a byproduct of Algorithm 1.

In our algorithm, consider what happens to the argument of the term $e_{-\a_k}(\cdot)$ as $e_{-\a_k}(x_k)$ is pushed passed all other terms in $p_{k+1}$ to the right hand side of the product, where it emerges as $e_{-\a_k}(y_k)$. Apart from the initial step where a constant (that is independent of $x_k,\ldots,x_N$) is added to the argument, the only time when the argument of $e_{-\a_k}(\cdot)$ changes is when it moves past a term of the form $h_{-\beta}(w_\beta)$, which has the effect of multiplying the argument by $w_k^{-\langle\a_k,\beta^\vee\rangle}$. Here we are using (\ref{ehcommute}) as well as Proposition \ref{diagonalterms}. Necessarily, we must have that $\beta>_\ii\a$ for this to occur.

Thus we have
$$
dy_k=\left(\prod_{i=k+1}^N |w_{\a_i}|^{-\langle\a_k,\a_i^\vee\rangle}\right) dx_k.
$$

Multiplying over all $k$ yields
$$
dy_1\ldots dy_N=\left( \prod_{i=1}^N q^{ -m_i\sum_{j=1}^{i-1} \langle\a_j,\a_i^\vee\rangle } \right)dx_1\ldots dx_N.
$$

Thus the proposition follows from the following lemma.
\end{proof}

\begin{lemma}
For all $\beta\in \Phi^+$, we have
$$ \langle\rho,\beta^\vee\rangle=1+\sum_{\a<_\ii\beta}\langle\a,\beta^\vee\rangle.$$
\end{lemma}
\begin{proof}
We first show that the right hand side of the above equation is independent of the choice of $\ii\in\I$. To achieve this, it suffices to show invariance under changing $\ii$ to $\ii'$ by the application of a single braid relation. For such $\ii$ and $\ii'$, we have that $\a_j(\ii)=\a_j(\ii')$ (here we are merely expressing the dependence of $\a_j$ on the choice of $\ii$) unless $\ii_j\neq \ii'_j$ so we have quickly reduced to showing invariance of long word decomposition in the rank 2 case. This is achieved by a simple case by case analysis, which we shall omit here.

Now consider two long word decompositions $\ii=(i_1,\ldots,i_N)$ and $\ii'=(\tau(i_N),\ldots,\tau(i_1))$ where $\tau\map{I}{I}$ is the automorphism defined by $w_0\a_i=-\a_{\tau(i)}$.

This pair of decompositions has the property that for any two $\a,\beta\in\Phi^+$, $\a<_\ii\beta$ if and only if $\beta<_{\ii'}\a$. Thus
\begin{eqnarray*}
1+\sum_{\a<_\ii\beta}\langle\a,\beta^\vee\rangle&=&\frac{1}{2}\left(1+\sum_{\a<_\ii\beta}\langle\a,\beta^\vee\rangle+1+
\sum_{\a<_{\ii'}\beta}\langle\a,\beta^\vee\rangle\right) \\
&=& \frac{1}{2}\sum_{\a\in\Phi^+}\langle\a,\beta^\vee\rangle \\
&=& \langle\rho,\beta^\vee\rangle
\end{eqnarray*} as required.
\end{proof}

Recall that $n$ denotes the degree of the cover $\G\rightarrow G$. For each root $\a$, define the integer $n_\alpha=\frac{n}{(n,l_\a)}$, where $l_\a$ is, as before given by $||\a^\vee||^2$, where the norm on the coroot lattice has been chosen such that the short coroots have length 1.  We also denote $n_k=n_{\a_k}$. The following lemma will allow us to integrate the function $f$ from (\ref{fdefn}) over any set $C_\m^\ii$.

\begin{lemma}\label{lemma4}
\begin{enumerate}
\item For the function $f$ defined in (\ref{fdefn}), $$\int_{C_\m^\ii}f(u) du=0$$ unless $n_k|m_k$ for all $k\in[N]$.

\item If $n_k| m_k$ for all $k\in[N]$ then $f$ is constant on $C_\m^\ii$ and takes the value
$$\prod_{\a\in\Phi^+}(q^{\langle\rho,\a^\vee\rangle}x_\a)^{m_\a}$$
where the variables $x_\a$ are defined by $x_\a=\prod_{i\in I}x_i^{h_i}$ with $\a^\vee=\sum_{i\in I}h_i\a_i^\vee$.
\end{enumerate}
\end{lemma}
\begin{proof}
By induction on $k$, we shall prove that
$$
\int_{C_\m^\ii}f(u)du\neq 0\implies n_i|m_i \text{ for all } i\leq k
$$
and in this case
$$f(\prod_{i=N}^1 h_{-\a_i}(w_i))=f\left(\prod_{i=N}^{k+1} h_{-\a_i}(w_i)\prod_{i=k}^1 h_{\a_i}(\p^{m_i})\right).$$

Note that the lemma follows from the $k=N$ case of the inductive statement.

Let $w_i=\p^{-m_i}u_i$. Then $u_i\in O_F$. We may write,
assuming our inductive hypothesis for $k$ and using (\ref{hhcommute}),
\[
h_{-\a_{k+1}}(w_{k+1})\prod_{i=k}^1 h_{\a_i}(\p^{m_i})=(\p,u_{k+1})^e\left( \prod_{i=k+1}^1 h_{\a_i}(\p^{m_i})\right) h_{-\a_{k+1}}(u_{k+1}),
\]
where the exponent $e$ is given by
\begin{equation}\label{exponent}
e=m_{k+1}l_{k+1}+\sum_{i\leq k} \langle\a_i,\a_{k+1}^\vee\rangle m_i l_{\a_i}.
\end{equation}

Consider the evaluation of the integral of $f$ over $C^\ii_\m$, by integrating over variables $y_1,\ldots,y_n$ in that order.
Suppose also that $m_{k+1}>0$. Then the above calculation shows that we get a factor of
$$ \int_{O_F^\times}(\p,u)^e du$$ appearing in our integral, from the integration over $y_{k+1}$. For this integral to be non-zero, we must thus have that $n|e$. By our inductive hypothesis we have that $e\equiv m_{k+1}l_{k+1}\pmod n$, so we obtain the first part of the inductive hypothesis, namely that $n_{k+1}$ divides $m_{k+1}$. The second part of the inductive hypothesis now follows from the form taken by $f$ in (\ref{fdefn}), noting that $e\equiv 0\pmod n$ and $h_{-\a_{k+1}}(u_{k+1})\in K$, thus proving our lemma.
\end{proof}

At this point it is straightforward to complete the evaluation of the integral of $f$ over $U^-$. In order for this integral to be convergent, we must assume that $|x_i|<1$ for all $i\in I$. However some sense can be made of this integral for an arbitrary collection of complex numbers $x_i$ by considering the integral as the analytic continuation of an intertwining operator, as in \cite{kp},\cite{mcn}. We obtain the following result, where we assume $|x_\a|<1$ for all $\a$ in order to ensure convergence.

\begin{theorem}\label{gk}
[Gindikin-Karpelevich Formula] 
For the function $f$ defined in (\ref{fdefn}), and the variables $x_\a$ as defined in Lemma \ref{lemma4}, we have \[
\int_{U^-}f(u) du=\prod_{\a\in\Phi^+} \left( \frac{1-q^{-1}x_\a^{n_\a}}{1-x_\a^{n_\a}} \right)
\]
\end{theorem}
\begin{proof}
The idea is to simply write the integral as a sum over each $C_\m^\ii$, and use what we have already computed.
\begin{eqnarray*}
\int_{U^-}f(u)du&=&\sum_{\m\in\N^N}\int_{C_\m^\ii} f(u)du
\ =\,\sum_{\substack {\m\in\N^N \\ n_j|m_j\forall j}}\int_{C_\m^\ii} f(u)du \\
&=&\sum_{\substack {\m\in\N^N \\ n_j|m_j\forall j}}\prod_{\a\in\Phi^+} (q^{-\langle\rho,\a^\vee\rangle}x_\a)^{m_\a} \prod_{i=1}^N q^{\langle\rho,\a_i^\vee\rangle m_i}(1-\frac{1-\delta_{m_i,0}}{q}) \\
&=& \prod_{\a\in\Phi^+}\sum_{k=1}^\infty x_\a^{kn_\a}(1-\frac{1-\delta_{k,0}}{q}) \\
&=& \prod_{\a\in\Phi^+} \left( \frac{1-q^{-1}x_\a^{n_\a}}{1-x_\a^{n_\a}} \right)
\end{eqnarray*} as required.\end{proof}

This same method can be used to calculate the integral of $f$ over other unipotent subgroups of $U^-$. For each $w\in W$ let $\Phi_w=\{\a\in\Phi^+\mid w\a\in \Phi^-\}$ and $U_w^-$ be the corresponding unipotent subgroup. Then with exactly the same proof, we have

\begin{theorem}\label{gkw}
For the function $f$ defined in (\ref{fdefn}), and the variables $x_\a$ as defined in Lemma \ref{lemma4}, we have \[
\int_{U_w^-}f(u) du=\prod_{\a\in\Phi_w} \left( \frac{1-q^{-1}x_\a^{n_\a}}{1-x_\a^{n_\a}} \right)
\]
\end{theorem}

From a representation-theoretic point of view, this integral corresponds to the evaluation of the intertwining operator $T_w\in\Hom(V_\chi,V_\chi^w)$ at a spherical vector.

\section{Connection to the combinatorics of crystals}\label{xtal}

In this section we shall give a bijection between the sets $C_\m^\ii$ that decompose $U^-$ and Lusztig's canonical basis. We shall be considering the case $F=k((\p))$ for $k$ an algebraically closed field, since this will make our exposition easier. Under this assumption, our sets $C_\m^\ii$ are naturally affine varieties over $k$.

Although such a choice of $F$ is not a local field, we can obtain a direct link between the case of a positive characteristic local field, which is necessarily (non-canonically) isomorphic to $\F_q((\p))$. The sets $C_\m^\ii$ for the field $\F_q((\p))$ are the $\F_q$ points of the corresponding varieties over $\overline{\F_q}$ that are considered in this section. To establish a link in the case for an arbitrary discrete valuation field $F$ with the techniques of this section, we note that a formal analogy can be made as the combinatorics of taking valuations is independent of the field $F$, and is an analogue of the process of tropicalisation.

 Lusztig \cite{lusztig} realised the canonical basis as the set of connected components of a graph structure on $\I\times\N^N$, which we shall now recall. In order for there to exist an edge between two vertices $(\ii,\m)$ and $(\ii',\m')$ it is first necessary to have $\ii$ and $\ii'$ related by a single application of a braid relation, so an occurrence of $(i,j,\ldots)$ is replaced by $(j,i,\ldots)$ when moving from $\ii$ to $\ii'$. Furthermore, we require that $\m$ and $\m'$ are related by piecewise linear transition maps $R_\ii^{\ii'}$, which we shall now explicitly describe. These equations appear in \cite{lusztig,kamnitzer} except for the case of type $G_2$, where they are derived from Section 7 of \cite{genminors}.

The piecewise linear transition maps $R_\ii^{\ii'}$ only affect the coordinates which are changed in moving from $\ii$ to $\ii'$. The restriction to the set of such coordinates is a local transtition map $R_{ij\ldots}^{ji\ldots}$, which we shall now give. There are four cases to consider, corresponding to each of the four finite type rank 2 root systems.

\begin{enumerate}
\item Type $A_1\times A_1$: If $a_{ij}=a_{ji}=0$, then
$$
R_{ij}^{ji}(a,b)=(b,a).
$$

\item Type $A_2$: If $a_{ij}=a_{ji}=-1$, then
$$
R_{iji}^{jij}(a,b,c)=(b+c-\min(a,c),\min(a,c),b+a-\min(a,c)).
$$

\item Type $B_2$: If $a_{ij}=-2$ and $a_{ji}=-1$, then
$$
R_{ijij}^{jiji}(a,b,c,d)=(b+2c+d-q,q-p,2p-q,a+b+c-p)
$$
where $p=\min(a+b,a+d,c+d)$ and $q=\min(2a+b,2a+d,2c+d)$.

\item Type $G_2$: If $a_{ij}=-3$ and $a_{ji}=-1$ then
\begin{multline*}
R_{ijijij}^{jijiji}(a,b,c,d,e,f)=(b+3c+2d+3c+f-r,r-q,2q-r-s,s-p-q, \\ 3p-s,a+b+2c+d+e-p)
\end{multline*}
where $p=\min(a+b+2c+d,a+b+2c+f,a+b+2e+f,a+d+2e+f,b+d+2e+f)$,

$q=\min(2a+2b+3c+d,2a+2b+3c+f,2a+2b+3e+f,2a+2d+3e+f,2c+2d+3e+f,a+b+d+2e+f+\min(a+c,2c,c+e,a+e)),$

$r=\min(3a+2b+3c+d,3a+2b+3c+f,3a+2b+3e+f,3a+2d+3e+f,3c+2d+3e+f,2a+b+d+2e+f+\min(a+c,2c,c+e,a+e)),$ and

$s=\min(2a+2b+2c+d+\min(a+b+3c+d,a+b+3c+f,a+b+3e+f,d+2e+f+\min(a+c,2c,c+e,a+e))+2f+3\min(a+b+2c,a+b+2e,a+d+2e,c+d+2e)).$
\end{enumerate}
We shall refer to this graph as the Lusztig graph.

The set of connected components of the Lusztig graph is canonically identified with the canonical basis $B(-\infty)$ for $U_q(\mathfrak{n}^+)$, the positive part of the quantised universal enveloping algebra, as defined for example in \cite{kashiwarasurvey}. Moreover, the following theorem from \cite[Ch 42]{lusztigbook} yields a parametrisation of the canonical basis for each choice of $\ii\in \I$.

\begin{theorem}[\cite{lusztigbook}, Ch 42\label{cellsascrystal}]
For each $\ii\in\I$, the map $\N^N\rightarrow\I\times\N^N$ sending $\m$ to $(\ii,\m)$ is a bijection between $\N^N$ and the set of connected components of the Lusztig graph. \end{theorem}

We consider $U^-$ as an ind-variety over $k$, and thus are able to give it the Zariski topology. The cells $C^\ii_\m$ now have the additional structure as subvarieties of $U^-$ and we denote their Zariski closure in $U^-$ by $\overline{C^\ii_\m}$. We are able to relate our decomposition of $U^-$ with the canonical basis via the following theorem.

\begin{theorem}[Theorem 15, \cite{baumanngaussent}, Theorem 4.5, \cite{kamnitzer}]\label{maincrystalthm}
For $(\ii,\m),(\ii',\m')\in\I\times\N^N$, the vertices $(\ii,\m)$ and $(\ii',\m')$ of the Lusztig graph are connected if and only if $\overline{C^\ii_\m}=\overline{C^{\ii'}_{\m'}}$.
\end{theorem}

The proof of this theorem utilises the circle of ideas presented in \cite{genminors}, which tells one how the coordinates on $U^-$ change when one changes the reduced word. This reduces the problem to the rank two case where it is dealt with by an explicit computation.

\begin{corollary}
Consider the collection of closed subsets of $U^-$ consisting of the set of all $\overline{C^\ii_\m}$ for a fixed $\ii$ with $\m$ running over $\N^N$. Then this collection is independent of the choice of $\ii$.
\end{corollary}

To proceed, we now recall some of the geometry of the affine Grassmannian for the adjoint group $G^\text{ad}$. This affine Grassmannian shall be denoted $\gr$, it is defined to be the fpqc quotient $G^{\rm{ad}}(k((\p)))/G^{\rm{ad}}(k[[\p]])$. However, it will suffice for our purposes to consider $\gr$ as an ind-variety over $k$. Recall that $\La$ is the cocharacter group of the maximal torus $T^\text{ad}$ of $G^\text{ad}$. The element $\p$ is a uniformiser, accordingly we use $\p^\la$ for $\la\in \La$ to denote the image of $\p$ in $T^\text{ad}$ under the cocharacter $\la$.

 The cells $C^\ii_\m$ that provided the decomposition of $U^-(F)$ are each (non-Noetherian) affine varieties over $k$. This is because each $C^\ii_\m$ is isomorphic to a product of copies of $k[[\p]]$ and $k[[\p]]^\times$, each of which are irreducible varieties. In particular, we see that $C^\ii_\m$ is irreducible.

We first consider the action of the group $G(O_F)$ on $\gr$ by left multiplication. The orbits are indexed by dominant weights $\la\in \La$. More precisely they are given by $\gr_\l=G(O_F)\p^\la$ according to the Cartan decomposition. The orbit $\gr_\la$ is finite dimensional and its closure $\overline\gr_\la$ is a (generally singular) projective variety.

We next consider the action of the unipotent groups $U^\pm(F)$ on $\gr$, again by left multiplication. The orbits here are given by the Iwasawa decomposition, indexed by elements of $\La$. For $\mu\in \La$, we denote by $S^{\pm}_\mu$ the orbit $U^\pm(F)\p^\mu$. These strata are locally closed, but not finite dimensional.

Fix a dominant weight $\la$. Then there is a set of Mirkovic-Vilonen cycles (hereafter MV cycles) $\mathcal{L}(\la)$, defined to be the irreducible components of $\overline{S^+_\mu\cap\gr_\la}$, where $\mu$ can be any weight. The only weights for which this intersection is non-empty are those for which $\mu$ belongs to the convex hull of the set  $\{w\la\mid w\in W\}$, and such that $\la-\mu$ is an integral sum of coroots. As a consequence, we deduce that $\mathcal{L}(\la)$ is a finite set.

There is also a notion of a stable MV cycle which we now define. Let $\mathcal{L}$ denote the set of all irreducible components of $\overline{S_\la^+\cap S_\mu^-}$, where $\la$ and $\mu$ run over $X$ with $\la\geq \mu$. This set $\mathcal{L}$ is equipped with an action of $\La$, given by $\nu Z=t^{\nu}Z$. We define the set of stable MV cycles to be the quotient of $\mathcal{L}$ by this action, and denote it by $\mathcal{L}(-\infty)$. We shall frequently identify $\mathcal{L}(-\infty)$ with a particular choice of coset representative in $\mathcal{L}$, most often that which is an irreducible component of $\overline{S_\la^+\cap S_0^-}$.

The following proposition of Anderson relates the notion of a MV cycle to that of a stable MV cycle.

\begin{proposition}[Proposition 3, \cite{anderson}]\label{anderson}
Suppose $\la$ is antidominant and $\nu$ is arbitrary. Then the irreducible components of $\overline{S_\nu^+\cap \gr_\la}$ are the same as the irreducible components of $\overline{S_\nu^+\cap S_\la^-}$ contained in $\overline{\gr_\la}$.
\end{proposition}

In particular we are able to identify $\mathcal{L}(\la)$ as a subset of $\mathcal{L}(-\infty)$.

 Under the gometric Satake correspondence, the intersection cohomology complex of $\overline\gr_\l$ is mapped to the space of the representation of $(G^{\text{ad}})^{\vee}$ (the Langlands dual to $G^\text{ad}$) of highest weight $\la$. Since a basis for the intersection cohomology of $\overline\gr_\l$ is given by the set of MV cycles \cite{mv}, it is natural to ask whether we can equip this set of MV cycles with the structure of a crystal. Such a procedure is carried out by Braverman and Gaitsgory in \cite{bravermangaitsgory} for the finite crystal $\mathcal{L}(\la)$  and by Braverman, Finkelberg and Gaitsgory in \cite{bg2} for the infinite crystal $\mathcal{L}(-\infty)$. These two crystal structures are compatible with the inclusion $\mathcal{L}(\la)\hookrightarrow\mathcal{L}(-\infty)$.

The following theorem gives the relation between our cells and MV cycles.

\begin{theorem}\label{cellsaremvcycles}
$\overline{\phi_0(C^\ii_\m)}$ is a MV cycle. Furthermore this map defines a bijection between $\N^N$ and $\mathcal{L}(-\infty)$ for any choice of $\ii\in \I$.
\end{theorem}

Before we embark on a proof, we first pause to define some useful notation.

For $\mu\in \La$, define the map
$$\phi_\mu\map{U^-}{\gr}, \quad  \phi_\mu(u)=\p^\mu u.$$

For any $\m\in\N^N$ and a fixed choice of $\ii\in \I$, we define $\m^\vee=\sum_{j=1}^N m_j\a_{i_j}^\vee$.

For any $\la\in \Z\Phi^\vee$, let $d_\la$ be the number of ways of writing $\la=\sum_{\a\in\Phi^+}m_\a \a^\vee$ for non-negative integral $m_\a$. This is a well known function, commonly referred to as the Kostant partition function. It is equal to the dimension of the $\la$ weight space in a Verma module, and thus we know that it is equal to the number of elements of weight $\la$ in $\mathcal{L}(-\infty)$, or equivalently the number of irreducible components of $\overline{S_\la^+\cap S_0^-}$.

Now we present a proof of Theorem \ref{cellsaremvcycles}.

\begin{proof}
Fix $\la$, and consider only those $\m$ for which $\m^\vee=\la$. Note that $\phi_0(C^\ii_\m)\subset S^+_\la\cap S_0^-$. Being irreducible, its closure must lie in a stable MV cycle. Since the number of such $\m$ is equal to the number of such stable MV cycles (they are both equal to $d_\la$), and $\phi_0$ of the union of such $C^\ii_\m$ surjects onto $S^+_\la\cap S_0^-$, the closure must equal a stable MV cycle.
\end{proof}



A crystal $\mathcal{B}$ comes equipped with five functions from $\mathcal{B}$ which are required to satisfy various compatibility relations that can be found in \cite{kashiwarasurvey}. The only two of these functions we will need are the Kashiwara operator $\tilde{e}_i\map{B}{B\sqcup \{0\}}$ and the function $\epsilon_i\map{B}{\Z\sqcup\{-\infty\}}$. The crystal $B(-\infty)$ is a lowest weight crystal, it is generated by a single element $b_0\in B(-\infty)$ together with its successive images under the various Kashiwara operators $\tilde{e}_i$.

We shall require the following proposition from \cite{baumanngaussent} which identifies $\tilde{e}_i(Z)$ and $\epsilon_i(Z)$ for a MV cycle $Z$.

\begin{proposition}[Proposition 14, \cite{baumanngaussent}]\label{kashiwaraaction}
Let $Z$ be a MV cycle, and $i\in I$. Then for each $p\in O_F$, the action of $e_{-\a_i}(pt^{\epsilon_i(Z)})$ stabalises $Z$. The MV cycle $\tilde{e}_i(Z)$ is the closure of $$
\{ e_{-\a_i}(x)z\mid z\in Z\mbox{ and }x\in F^\times \mbox{ such that } v(x)=\epsilon_i(Z)+1\}.
$$
\end{proposition}



We will write $B(\la)$ and $B(-\infty)$ for the abstract crystals reaslised by $\mathcal{L}(\la)$ and $\mathcal{L}(-\infty)$ respectively. For $b\in B(-\infty)$, represented by $(\ii,\m)$, let $C_b=\overline{C_\m^\ii}$. We now move to studying the finite crystal $B(\la)$, for which our starting point is the following Proposition.

\begin{proposition}[Proposition 8.2, \cite{kashiwarasurvey}, \S 8 \cite{lusztig2}, Corollary 3.4, \cite{genminors}, Theorem 8.4, \cite{kamnitzer}]\label{ka}
The image of the natural injection $B(\la)\hookrightarrow B(-\infty)$ is given by
$$
B(\la)=\{b\in B(-\infty)\mid\epsilon_i(b)+\langle\alpha_i,\la\rangle\geq 0\ \forall\  i\in I\}.
$$
\end{proposition}

We use this to identify $B(\la)$ in terms of cells $C_b$. There is an alternative characterisation of the cells of the finite crystal $B(\la)$ which will prove to be useful in the following section. For $u\in U^-$ and $i\in I$, we define $f_i(u)$ to be element of $F$ such that in writing $u=\prod_{\a\in\Phi}e_{-\a}(z_\a)$, we have $z_{\a_i}=f_i(u)$. This result can be considered as an analogue to Theorems 8.3 and 8.5 of \cite{kamnitzer}.

\begin{theorem}\label{finitecrystal}
Suppose that $\la$ is dominant and write $\la'=w_0\la$. Then the following three conditions on an element $b$ of the crystal basis are equivalent.
\begin{enumerate}
\item $b\in B(\la)$
\item $\p^{\la} C_b \p^{-\la} \subset K$
\item For all $i\in I$ and $u\in C_b$, we have $f_i(\p^{\la} u\p^{-\la})\in O_F$.
\end{enumerate}
\end{theorem}

\begin{proof}
$1\Rightarrow 2$.
Proposition \ref{anderson} shows that $b\in B(\la)$ if and only if $\phi_{\la'}(C_b)\subset \overline{\gr_\la}$, where $\la'=w_0\la$ is the image of $\la$ under the action of the longest word in $W$. So if $b\in B(\la)$, we have $C_b\subset\phi^{-1}_{\la'}(\overline{\gr_\la})$. It is shown in \cite[4.4.4(ii)]{bruhattits} that $\phi^{-1}_{\la'}(\overline{\gr_\la})=\p^{-\la}K\p^\la\cap U^-$, completing this implication

$2\Rightarrow 3$.
This is immediate.

$3 \Rightarrow 1$.
Suppose for all $i\in I$ and $u\in C_b$ that $f_i(\p^\la u \p^{-\la})\in O_F$. Then $x_i(u)\in \p^{-\langle\a_i,\la\rangle}O_F$. So utilising Proposition \ref{kashiwaraaction}, we have that the action of $e_{-\a_i}(\p^{\epsilon_i(b)}O_F)$ stabalising $b$ implies that $\epsilon_i(b)\geq -\langle\a_i,\la\rangle$ which by the characterisation of $B(\la)$ as a subset of $B(-\infty)$ in Proposition \ref{ka} tells us that $b\in B(\la)$.

\end{proof}

\section{Computation of a Whittaker Function in Type A}\label{tok}

We conclude by presenting a sample calculation of a metaplectic Whittaker function, and show that in this case, it is equal to the $p$-part of a Weyl Group Multiple Dirichlet Series, as defined and studied in \cite{wgmdsxtal}. This identity will be more than an equality of functions, as we will show that for our method of writing a Whittaker function as a weighted sum over a crystal, we will recover exactly the same weighted sum that appears in \cite{wgmdsxtal}. We begin by reproducing the combinatorial description of the $p$-part of a type $A_r$ multiple Dirichlet series.

A Gelfand-Tsetlin pattern is a triangular array of natural numbers
\[
\mathfrak{T}=\left\{
\begin{array}{cccccccccc}
a_{0,0} & & a_{0,1} & & a_{0,2} & & \ldots & a_{0,r-1} & & a_{0,r}  \\
 & a_{1,1} & & a_{1,2} & & a_{1,3} & \ldots &  & a_{1,r} & \\
 & & \ddots & & & & & \Ddots  & & \\
 & & & & & a_{r,r} & & & &
\end{array}
 \right\}
\]
subject to the inequalities $a_{i,j}\geq a_{i+1,j+1}\geq a_{i,j+1}$ for all $i$ and $j$.

To such a pattern, one attaches integers $e_{i,j}$ for all $1\leq i\leq j\leq r$ defined by
$$
e_{i,j}=\sum_{k=j}^r a_{i,k}-a_{i-1,k},
$$
and a weight
$$
G(\mathfrak{T})=\prod_{1\leq i\leq j\leq r} \gamma(a_{i,j}),\quad \gamma(a_{i,j})=\begin{cases} q^{e_{i,j}} &\text{if  } a_{i-1,j-1}>a_{i,j}=a_{i-1,j}, \\
g(e_{i,j},0) &\text{if  } a_{i-1,j-1}>a_{i,j}>a_{i-1,j}, \\
g(e_{i,j},-1) &\text{if  } a_{i-1,j-1}=a_{i,j}>a_{i-1,j}, \\
0 &\text{if  } a_{i-1,j-1}=a_{i,j}=a_{i-1,j}, \end{cases}
$$ where the Gauss sums $g(a,b)$ are as defined by (\ref{gausssumdefn}) (in Section \ref{prelim}). Furthermore, define integers
$$
k_i(\mathfrak{T})=\sum_{j=1}^r (a_{i,j}-a_{0,j}).
$$
Then the $p$-part of a type A multiple Dirichlet series is given by the following expression:

$$
\sum_{\mathfrak{T}}G(\mathfrak{T}) \, 
q^{-2(k_1s_1+\ldots +k_rs_r)}
$$
where the sum is over all Gelfand-Tsetlin patterns with a fixed top row and $s_1,\ldots,s_r\in\C$ is a collection of complex variables. 

We shall work with the simple group $G=SL_{r+1}$ with the nicest possible choice of the long word decomposition. Explicitly realise
$$\Phi^+=\{ (i,j)\in [r+1]\times [r+1]\mid i<j \}.$$
Write $s_i$ for the simple reflection corresponding to the simple root $(i,i+1)$. Then we choose $\ii\in\I$ corresponding to the long word decomposition $w_0=s_1(s_2s_1)(s_3s_2s_1)\ldots(s_rs_{r-1}\ldots s_2s_1)$. We shall refer to this word as the Gelfand-Tsetlin word. This choice of word determines the ordering $<_\ii$ which can be explicitly stated in the form $(i,j)<_\ii (i',j')$ if $i<i'$ or if $i=i'$ and $j<j'$.

Note that a particular feature of the Gelfand-Tsetlin word is that in performing Algorithm \ref{algorithm}, the Iwasawa decomposition for the lower right copy of $SL_r$ in $SL_{r+1}$ is performed first as a substep of performing the Iwasawa decomposition for $SL_{r+1}$. This enables us to perform induction on $r$, a feature we shall frequently exploit below.

For a more in depth study of our main algorithm, we shall explitly realise $G$ as a $(r+1)\times (r+1)$ matrix group in the usual way. Thus $e_{(i,j)}(x)$ has an $x$ appearing in the entry of the $i$-th row and $j$-th column, ones along the diagonal and zeroes elsewhere.

Firstly, we shall require further information regarding the top row of the matrix $p_1$ that appears in our Iwasawa decomposition.

\begin{lemma}
 The entries in the top row of $p_1$ are given by
 \[
  (p_1)_{1j} =
\begin{cases}
  \prod_{i=j}^{r}\frac{1}{w_{1,j+1}},  & \mbox{if }|y_{1,j}|>1 \mbox{ or } j=1\\
  0, & \mbox{otherwise.}
\end{cases}
 \]
\end{lemma}
\begin{proof}
The proof uses Proposition \ref{diagonalterms}, which in particular proves the lemma in the $j=1$ case, so we may assume $j>1$. By induction on $r$, we know $(p_{r+1})_{1j}=0$.
Note that the only time that the $i,j$-entry of the $p$-matrix in the algorithm is altered is when a multiplication by $h_{1,j}(y_{1,j}^{-1})e_{1,j}(y_{1,j})$ occurs.
Hence, $(p_1)_{ij}=(p_j)_{11}$ if $|y_{i,j}|\geq 1$ and is zero otherwise. Thus, the lemma follows from Proposition \ref{diagonalterms}.
\end{proof}

We shall remark that the above lemma, when applied to $SL_r$, yields the values of the entries of the second row of the matrix $p_{r+1}$.

We will write $\la$ as $\la=\sum_i \la_i\p_i$ where the $\p_i$ are the fundamental weights. Then the value of the character $\psi_\la$ is given by
$$
\psi_\la(u)=\prod_{i=1}^r \psi(\p^{-\la_i}x_{i,i+1}).
$$

Now we come to the formula for $\psi(u)$ in terms of the $w$ and $y$ variables.
Let \[
        \Psi_{\la}^{i,j}(u)=\begin{cases}
                   \psi\left(y_{i,j}\,\p^{-\la_i}\displaystyle\prod_{k=j+1}^{r+1}\frac{w_{i,k}}{w_{i+1,k}}\right) & \mbox{if }j=i+1\mbox{ or }|y_{i+1,j}|>1 \\
		   1 &\mbox{otherwise}.
                  \end{cases}
       \]

\begin{proposition}\label{archaracterformula}
We have the following formula for the value of the additive character $\psi$ on $U^-$, in the $w_\a$ coordinate system.
 $$\psi_{\la}(u)=\prod_{\a\in\Phi^+} \Psi^\a_\la(u).$$
\end{proposition}

\begin{proof}
By induction on $r$, it suffices to consider the case where $i=1$. Since we know that $\psi_\la(u)=\prod_{i}\psi(x_{i,i+1}\p^{-\la_i})$, we have to keep track of the $2,1$-entry of the matrix $e_{-\a_1}(x_1)\ldots e_{-\a_k}(x_k)p_k$ in order to express $x_{1,2}$ in terms of the $w$ or $y$ variables. The following sequence of events happens to this entry (note that we only need to consider $k\leq r$).

Initially, the (2,1) entry is $x_{1,2}$. For $k$ running from $r+1$ to 3, the following happens.

First one subtracts $y_{1,k}(p_{r+1})_{2,k}$, then divides by $w_{1,k}$.

At the end of this process, one is left with $y_1$, and hence has a formula for $x_1$ in terms of the $w$ variables. Then a simple application of the fact that $\psi$ is an additive character, together with our knowledge of $(p_{r+1})_{2,k}$ obtained via the previous lemma produces our desired formula for $\psi_\la(u)$.
\end{proof}

With the help of the calculation of the previous lemma, as well as the characterisation of the finite crystals from Theorem \ref{finitecrystal} from the previous section, we are able to immediately deduce the following explicit description of $B(\la+\rho)$ in the case under consideration. The weight $\rho$ is as usual half the sum of the positive coroots, in terms of the fundamental weights, we have $\rho=\p_1+\ldots +\p_r$.

\begin{proposition}
We have $(\ii,\m)\in B(\la+\rho)$ if and only if the following inequalities hold.

For each $(i,j)\in\Phi^+$, $m_{i,j}\geq 0$ and
\begin{equation}\label{crystalinequalities}
\sum_{k=j}^{r+1} m_{i,k}\leq \lambda_{i}+1+\sum_{k=j}^{r}m_{i+1,k+1}.
\end{equation}
\end{proposition}

In order to describe a combinatorial formula for the Whittaker function, we first decorate the tuple of integers $\m$ by circles and boxes.
Suppose $(\ii,\m)\in B(\la+\rho)$. For each $\a\in\Phi^+$, we define a weighting $w(\m,\a)$ as follows. Say that $m_\a$ is circled if $m_\a=0$ and boxed if equality holds in the corresponding inequality in (\ref{crystalinequalities}), ie if $\a=(i,j)$ and $\sum_{k=j}^{r+1} m_{i,k}=\lambda_{i}+1+\sum_{k=j+1}^{r+1}m_{i+1,k+1}$. (This vocabulary is used to match that of \cite{wgmdsxtal}). Corresponding to this decorating, we define a weight function
\begin{equation}\label{wdefn}
w(\m,\a)=\begin{cases}
\frac{q-1}{q^2}g(r_\a,s_\a)\ & \text{if $m_\alpha$ is not circled }  \\
1 & \text{if $m_\a$ is circled but not boxed} \\
0 & \text{if $m_\a$ is both boxed and circled}  \end{cases}
\end{equation}
where the integers $r_\a$ and $s_\a$ are defined by
\[
r_{i,j}=\sum_{k\leq i} m_{k,j}
\] and
$$
s_{i,j}=\la_{i}+\sum_{k=j}^{r}m_{i+1,k+1}-\sum_{k=j}^{r+1}m_{i,k}.
$$
Note that in the case where $m_\a$ is neither boxed nor circled, then Proposition \ref{gauss} on evaluating such Gauss sums can be used to simplify the resulting expression, which only depends on the value of $r_\a$ modulo $n$.

\begin{theorem}
The integral over $C_\m^\ii$ is given by
\[
\int_{C_\m^\ii}f(u)\psi_\la(u)du=\begin{cases}
\prod_{\a\in\Phi^+}w(\m,\a)x_\a^{m_\a} & \text{if } (\ii,\m)\in B(\la+\rho)  \\
0 & \text{otherwise}   \end{cases}
\]
\end{theorem}
\begin{remark}
We expect that the vanishing of the above integral over $C_\m^\ii$ should occur for the case of an arbitrary root system and choice of $\ii\in I$, whenever $(\ii,\m)\notin B(\la+\rho)$. This would yield an expression for the metaplectic Whittaker function as a sum over $B(\la+\rho)$ as opposed to the $B(-\infty)$ that our expression is currently a priori a sum over.
\end{remark}
\begin{proof}
We will first present the proof in the case where $m_\a>0$ for all $\a\in\Phi^+$, then discuss the changes that need to be made in order to incorporate the general case.

Recall that we have the variables $w_\a=\p^{-m_\a}u_\a$. In terms of such variables, we have
$$
\psi(u)=\prod_{(i,j)\in\Phi^+} \psi\left( \p^{s_{ij}}\frac{\prod_{k=j}^{r+1}u_{i,k}}{\prod_{k=j}^r u_{i+1,k+1}} \right)
$$
whereas
$$
f(u)= \prod_{\a\in\Phi^+}(q^{-\langle\rho,\a^\vee\rangle}x_\a)^{m_\a}(u_\a,\p)^{m_\a+\sum_{\a<_\ii \beta}\langle\beta,\a^\vee\rangle m_\beta}.
$$

We will show that the change of variables
$$
t_{i,j}=\frac{\prod_{k=j}^{r+1}u_{i,k}}{\prod_{k=j}^r u_{i+1,k+1}}
$$
transforms the integral $\int_{C_\m^\ii}f(u)\psi_\la(u)du$ into the product of integrals over $O_F^\times$ which are the defining integrals for the Gauss sums $g(r_\a,s_\a)$, that is

$$
\int_{C_\m^\ii}f(u)\psi_\la(u)du=\prod_{\a\in\Phi^+}(q^{-\langle\rho,\a^\vee\rangle}x_\a)^{m_\a}\int_{O_F^\times} (t_\a,\p)^{r_\a}\psi(\p^{s_\a}t_\a)dt_\a.
$$

Note that in this equation, the normalisation of Haar measure on $O_F^\times$ differs from the normalisation chosen in the introduction, which results in the appearance of the factor $\frac{q-1}{q^2}$ in the definition of $w(\m,\a)$.

To see this it suffices to check that
$$
\prod_{\a\in\Phi^+}u_\a^{m_\a+\sum_{\a<_\ii \beta}\langle\beta,\a^\vee\rangle m_\beta}=\prod_{(i,j)\in\Phi^+}t_{i,j}^{r_{i,j}}.
$$

A term $u_{i,j}$ on the product in the right hand side of the above product appears with a positive exponent in $t_{i,k}$ for $k\leq j$ and with a negative exponent in $t_{i-1,k}$ for $k\leq j-1$. When we pair up and cancel the occurrences from these terms as much as possible, we are left with $u_{ij}^{m_{ik}}$ from these terms. There also occurs $u_{ij}$ with exponent $\sum_{k\leq j}m_{k,j}$ from the $t_{i,j}$ term, and an occurrence of $u_{ij}$ with exponent $-\sum_{k\leq i-1} m_{k,i}$ from the $t_{i-1,i}$ term. Combining these gives exactly the formula we want, which is enough to prove the theorem in the case where all $m_\a$ are positive.

Now we shall consider the alterations to the above that need to be considered when some of the $m_\a$ are allowed to take the value zero.

In general, we perform the same manipulations to get a product of integrals of $(t_{i,j},\p)^{r_{i,j}}\psi(t_{i,j}\p^{s_{i,j}})$ except in two cases. The first of these is when $m_{i,j}=0$ in which case the term $(t_{i,j},\p)^{r_{i,j}}$ is no longer there, so we are left with simply an integral of an additive character over $O_F$, which we can easily evaluate to $w(\m,\a)$.

The more subtle case to handle is when we have $m_{i+1,j}=0$, for then this forces the `disappearance' of the additive character in the integrand. By Proposition \ref{gauss}, this is not of concern unless $s_{i,j}\leq-1$, for otherwise we have an explicit evaluation of the Gauss sums so we can see that we get our desired integral. Now note that the condition $m_{i+1,j}=0$ implies that $s_{i,j}=s_{i,j-1}+m_{i,j-1}$. Since $m_{i,j-1}\geq 0$, we have $s_{i,j-1}\leq -1$. These inequalities imply that the integral over $t_{i,j-1}$ is zero from our previous work, unless $m_{i+1,j-1}=0$. We can deal with this latter case because we may assume without loss of generality that $m_{i+1,j-1}\neq 0$ (if it exists) by decreasing $j$ if necessary.

Thus in all cases, we are lead to an easily evaluated integral, yielding the desired product formula for the integral over $C_\m^\ii$, so we are done.
\end{proof}

We have thus proven the following theorem.

\begin{theorem}\label{main}
The value of the integral $I_\la$ defined in (\ref{iladefn}) which calculates the metaplectic Whittaker function is given by
$$
I_\la=\sum_{(\ii,\m)\in B(\la+\rho)}\prod_{\a\in \Phi^+}w(\m,\a)x_\a^{m_\a}.
$$ where the weight $w(\m,\a)$ is as defined in (\ref{wdefn}).
\end{theorem}

Noting that the sum over the crystal $B(\la+\rho)$ is identical to the corresponding sum over the crystal expressing the $p$-part of a Weyl group multiple Dirichlet series as presented in \cite{wgmdsxtal}, we obtain the following corollary

\begin{corollary}
The coefficients of the $p$-part of a Weyl group multiple Dirichlet series in type A is equal to the value of a Whittaker function evaluated at a torus element on the corresponding local metaplectic group.
\end{corollary}


\end{document}